\documentclass[oneside,english]{amsart}
\usepackage[latin9]{inputenc}
\usepackage{amsthm}
\usepackage{amssymb}
\usepackage{wasysym}

\makeatletter
\numberwithin{equation}{section}
\numberwithin{figure}{section}
\theoremstyle{plain}
\newtheorem{thm}{\protect\theoremname}
\theoremstyle{plain}
\newtheorem{lem}[thm]{\protect\lemmaname}
\theoremstyle{plain}
\newtheorem{cor}[thm]{\protect\corollaryname}

\@ifundefined{date}{}{\date{}}
\numberwithin{thm}{section}
\usepackage{xcolor}

\makeatother

\usepackage{babel}
\providecommand{\corollaryname}{Corollary}
\providecommand{\lemmaname}{Lemma}
\providecommand{\theoremname}{Theorem}

\begin{document}
\title{On the convergence of Denjoy-Wolff points}
\author{Serban Belinschi, Hari Bercovici, and Ching-Wei Ho}
\address{Institut de Math\'ematiques de Toulouse: UMR5219, Université de Toulouse,
CNRS; UPS , F-31062 Toulouse, France}
\email{Serban.Belinschi@math.univ-toulouse.fr}
\address{Mathematics Department, Indiana University, Bloomington, IN 47405,
USA}
\email{bercovic@indiana.edu}
\address{Institute of Mathematics, Academia Sinica, Taipei 10617, Taiwan; Department
of Mathematics, University of Notre Dame, Notre Dame, IN 46556, United
States}
\email{cho2@nd.edu}
\begin{abstract}
If $\varphi$ is an analytic function from the unit disk $\mathbb{D}$
to itself, and $\varphi$ is not a conformal automorphism, we denote
by $\lambda_{\varphi}$ its Denjoy-Wolff point, that is, the limit
of the iterates $\varphi(\varphi(\cdots\varphi(0)\cdots))$. A result
of Heins shows that, given a sequence $(\varphi_{n})_{n\in\mathbb{N}}$
of such analytic functions that convergence pointwise to $\varphi$,
it follows that $\lim_{n\to\infty}\lambda_{\varphi_{n}}=\lambda_{\varphi}$.
This allows us to improve results about the contnuous extensions of
the subordination functions that arise in the study of free convolutions.
We also offer an alternate proof of the result of Heins.
\end{abstract}

\maketitle

\section{Introduction\label{sec:Introduction}}

Denote by $\mathbb{D}$ the unit disk in the complex plane $\mathbb{C},$
and let $\mathbb{T}$ be its boundary. Suppose that $\varphi:\mathbb{D}\to\overline{\mathbb{D}}$
is an analytic map. Denjoy and Wolff discovered independently \cite{denjoy,wolff}
that, if $\varphi$ does not have a fixed point in $\mathbb{D}$,
then it must have an attracting fixed point in $\mathbb{T}$. More
precisely, in this case, there exists a unique point $\lambda_{\varphi}\in\mathbb{T}$
with the property that
\[
\lim_{r\uparrow1}\varphi(r\lambda_{\varphi})=\lambda_{\varphi},
\]
and a certain nonnegative quantity, denoted $\varphi'(\lambda_{\varphi})$
and called the Julia-Carath\'eodory derivative of $\varphi$ at $\lambda_{\varphi}$,
exists and satisfies
\[
\varphi'(\lambda_{\varphi})\le1.
\]
 It follows that for an analytic map $\varphi:\mathbb{D}\to\overline{\mathbb{D}}$
one, and only one, of the following alternatives occurs.
\begin{enumerate}
\item $\varphi$ is the identity map $\varphi(z)=z$.
\item $\varphi(z)=\gamma$ for every $z\in\mathbb{D}$, where $\gamma\in\mathbb{T}$.
In this case, we set $\lambda_{\varphi}=\gamma$ and this point can
be thought of as a generalized fixed point for $\varphi$.
\item $\varphi$ is not the identity map but it has a fixed point $\lambda_{\varphi}\in\mathbb{D}$
such that $|\varphi'(\lambda_{\varphi})|=1$. In this case $\varphi$
is a (hyperbolic) rotation about the point $\lambda_{\varphi}$ and
it has no other fixed points in $\mathbb{D}$.
\item $\varphi$ has a fixed point $\lambda_{\varphi}\in\mathbb{D}$ and
$|\varphi'(\lambda_{\varphi})|<1$. In this case, $\lambda_{\varphi}$
is the limit of the iterates $\varphi(\varphi(\cdots(\varphi(z))\cdots))$
for every $z\in\mathbb{D}$.
\item $\varphi(\mathbb{D})\subseteq\mathbb{D}$ and $\varphi$ has no fixed
point in $\mathbb{D}$. In this case, the Denjoy-Wolff result provides
the point $\lambda_{\varphi}\in\mathbb{T}$ which is, again, the limit
of the iterates $\varphi(\varphi(\cdots(\varphi(z))\cdots))$ for
every $z\in\mathbb{D}$. The function $\varphi$ may have other fixed
points on $\mathbb{T}$, but the Julia-Carath\'eodory derivative
at such points will be greater than one.
\end{enumerate}
Thus, the point $\lambda_{\varphi}$ is defined for every analytic
selfmap of $\mathbb{D}$ other than the identity map. For our purposes,
Denjoy-Wolff points on $\mathbb{T}$ are best understood in the context
of the conformally equivalent domain
\[
\mathbb{H}=\{x+iy\in\mathbb{C}:y>0\}.
\]
This is mapped to $\mathbb{D}$ via the function $z\mapsto(z-i)/(z+i)$,
$z\in\mathbb{H}$. This identification extends to a homeomorphism
of the closure
\[
\overline{\mathbb{H}}=\mathbb{H}\cup\mathbb{R}\cup\{\infty\}
\]
in the Riemann sphere to
\[
\overline{\mathbb{D}}=\mathbb{D}\cup\mathbb{T}
\]
that sends $\infty$ to $1$. Thus, any point in $\mathbb{R}\cup\{\infty\}$
can be the Denjoy-Wolff point of some analytic map $\psi:\mathbb{H}\to\mathbb{H}.$
According to Nevanlinna \cite{nev}, an arbitrary such map $\psi$
can be written under the form
\[
\psi(z)=\beta+\int_{\mathbb{R}\cup\{\infty\}}\frac{1+tz}{t-z}\,d\sigma(t),\quad z\in\mathbb{H},
\]
where $\beta\in\mathbb{R}$ and $\sigma$ is a finite Borel measure
on the one-point compactification of $\mathbb{R}$. The fraction above
must, of course, be understood to be $z$ when $t=\infty$, so the
Nevanlinna formula can also be written as
\[
\psi(z)=\alpha z+\beta+\int_{\mathbb{R}}\frac{1+tz}{t-z}\,d\sigma(t),\quad z\in\mathbb{H},
\]
where $\alpha=\sigma(\{\infty\})$. Clearly, $\psi$ is the identity
map precisely when $\alpha=1$, $\beta=0$, and $\sigma(\mathbb{R})=0$.
We have $\lambda_{\psi}=\infty$ precisely when $\alpha\ge1$ and
$\psi$ is not the identity map. (The number $1/\alpha$ corresponds
to the Julia-Carath\'eodory derivative of $\psi$ at $\infty$.) 

The follwong result is proved in \cite{heins}. 
\begin{thm}
\label{thm:convergence theorem} Suppose that $(\varphi_{n})_{n\in\mathbb{N}}$
is a sequence of analytic self-maps of $\mathbb{D}$ that converges
pointwise to $\varphi:\mathbb{D\to\mathbb{D}},$ and suppose that
$\varphi$ is not the identity map. Then $\lim_{n\to\infty}\lambda_{\varphi_{n}}=\lambda_{\varphi}$.
\end{thm}

It is, of course, obvious that $\varphi_{n}$ is not the identity
map for large $n$, and then $\lambda_{\varphi_{n}}$ is indeed defined
eventually. Note also that the functions $\varphi_{n}$ converge to
$\varphi$ locally uniformly by the Vitali-Montel theorem (see, for
instance \cite{key-13}). An alternate proof of Theorem \ref{thm:convergence theorem}
is provided in Section \ref{sec:Proof-of-the convergence thm}. In
Section \ref{sec:applications}, we take advantage of the fact that
one can use another space of parameters in place of $\mathbb{N}$
to derive applications to the subordination functions that occur in
free probability. Many of these applications were known, though the
original proofs are more involved and rely on the functional equations
that subordination functions satisfy (see \cite{voic-fish1,serb-leb,serb-boundedness,bb-semigroups}).

We wish to thank Marco Abate for bringing the work of Heins \cite{heins},
as well as the existence of subsequent developments, to our attention.
Theorem \ref{thm:convergence theorem} was not previously known to
us and we included it as an original result. Our proof is perhaps
more elementary than the original and we decided to retain it in this
note.

\section{Proof of the convergence theorem\label{sec:Proof-of-the convergence thm}}

Throughout this section, we assume that $\varphi_{n}$ and $\varphi$
are given maps $\mathbb{D}\to\mathbb{D}$ that satisfy the hypotheses
of Theorem \ref{thm:convergence theorem}. Since $\overline{\mathbb{D}}$
is a compact metric space, the sequence $(\lambda_{\varphi_{n}})_{n\in\mathbb{N}}$
converges if and only if it has a unique subsequential limit. Therefore,
for the proof of Theorem \ref{thm:convergence theorem} we may assume
that the limit $\mu=\lim_{n\to\infty}\lambda_{\varphi_{n}}$ exists,
and we must prove that $\mu=\lambda_{\varphi}.$ We distinguish two
cases, according to whether $\mu\in\mathbb{D}$ or $\mu\in\mathbb{T}$. 
\begin{lem}
\label{lem:mu in disk}If $\mu\in\mathbb{D}$ then $\mu=\lambda_{\varphi}$.
\end{lem}

\begin{proof}
The functions $\varphi_{n}$ converge to $\varphi$ uniformly in some
neighborhood of $\mu$. Since $\lambda_{\varphi_{n}}$ belongs eventually
to that neighborhood, this implies that
\[
\mu=\lim_{n\to\infty}\lambda_{\varphi_{n}}=\lim_{n\to\infty}\varphi_{n}(\lambda_{\varphi_{n}})=\varphi(\mu).
\]
Thus $\mu$ is the (necessarily) unique fixed point of $\varphi$,
that is, $\lambda_{\varphi}=\mu$.
\end{proof}
\begin{lem}
\label{mu in the circle}If $\mu\in\mathbb{T}$ then $\mu=\lambda_{\varphi}$.
\end{lem}

\begin{proof}
We write
\[
\lambda_{\varphi_{n}}=r_{n}\mu_{n},\quad,n\in\mathbb{N},
\]
where $\mu_{n}\in\mathbb{T}$ converge to $\mu$, $r_{n}\in[0,1]$,
and $\lim_{n\to\infty}r_{n}=1$. Consider the new maps $\widetilde{\varphi}_{n},\widetilde{\varphi}:\mathbb{D}\to\mathbb{D}$
defined by
\[
\widetilde{\varphi}_{n}(\lambda)=\varphi_{n}(\mu_{n}\lambda)/\mu_{n},\quad\widetilde{\varphi}(\lambda)=\varphi(\mu\lambda)/\mu,\quad n\in\mathbb{N},\lambda\in\mathbb{D}.
\]
 We have then
\[
\lambda_{\widetilde{\varphi}}=\lambda_{\varphi}/\mu,\quad\lambda_{\widetilde{\varphi}_{n}}=r_{n},\quad n\in\mathbb{N},
\]
and the sequence $(\widetilde{\varphi}_{n})_{n\in\mathbb{N}}$ converges
pointwise to $\widetilde{\varphi}$. To conclude the proof, it suffices
to show that $\lambda_{\widetilde{\varphi}}=1$. At this point, we
use the conformal map
\[
u(z)=\frac{z-i}{z+i},\quad z\in\mathbb{H},
\]
to reformulate the problem. We define analytic maps $\psi_{n},\psi:\mathbb{H}\to\mathbb{H}$
by
\[
\psi_{n}(z)=u^{-1}(\widetilde{\varphi}_{n}(u(z))),\psi(z)=u^{-1}(\widetilde{\varphi}(u(z)))\quad z\in\mathbb{H}.
\]
The sequence $(\psi_{n})_{n\in\mathbb{N}}$ converges pointwise to
$\psi$. Since $\{u(iy):y>0\}=(-1,1)$, it follows that $\lambda_{\psi_{n}}$
is either $\infty$ or of the form $iy_{n}$ for some $y_{n}>0$,
and moreover $\lim_{n\to\infty}y_{n}=\infty$. We conclude the proof
by showing that $\lambda_{\psi}=\infty$. To do this, we write the
Nevanlinna representations
\begin{align*}
\psi_{n}(z) & =\beta_{n}+\int_{\mathbb{R}\cup\{\infty\}}\frac{1+tz}{t-z}\,d\sigma_{n}(t)=\alpha_{n}z+\beta_{n}+\int_{\mathbb{R}}\frac{1+tz}{t-z}\,d\sigma_{n}(t),\\
\psi(z) & =\beta+\int_{\mathbb{R}\cup\{\infty\}}\frac{1+tz}{t-z}\,d\sigma(t)=\alpha z+\beta+\int_{\mathbb{R}}\frac{1+tz}{t-z}\,d\sigma(t),\quad z\in\mathbb{H}.
\end{align*}
The convergence of $\psi_{n}$ to $\psi$ amounts to $\lim_{n\to\infty}\beta_{n}=\beta$
and to the convergence of $\sigma_{n}$ to $\sigma$ in the weak$^{*}$-topology
(obtained by viewing these measures as linear functionals on the Banach
space $C(\mathbb{R}\cup\{\infty\})$). Clearly, if $\lambda_{\psi_{n}}=\infty$
for infinitely many values of $n$, then we have $\alpha_{n}=\sigma_{n}(\{\infty\})\ge1$
for infinitely many values of $n$, and weak convergence implies $\alpha\ge1$
as well. We can therefore restrict ourselves to the case in which
$\lambda_{\psi_{n}}=iy_{n}$ with $y_{n}\in(0,+\infty).$ Thus, $iy_{n}$
is an actual fixed point of $\psi_{n}$, that is
\[
iy_{n}=\beta_{n}+\int_{\mathbb{R}\cup\{\infty\}}\frac{1+tiy_{n}}{t-iy_{n}}d\sigma_{n}(t).
\]
Taking imaginary parts in this equation (and dividing by $y_{n}$)
yields
\[
1=\int_{\mathbb{R}\cup\{\infty\}}\frac{1+t^{2}}{t^{2}+y_{n}^{2}}\,d\sigma_{n}(t).
\]
Suppose that $y_{n}\ge1$. For fixed $T>0$, we have
\begin{align*}
1 & =\int_{[-T,T]}\frac{1+t^{2}}{t^{2}+y_{n}^{2}}\,d\sigma_{n}(t)+\int_{[\mathbb{R}\cup\{\infty\}]\backslash[-T,T]}\frac{1+t^{2}}{t^{2}+y_{n}^{2}}\,d\sigma_{n}(t)\\
 & \le\int_{[-T,T]}\frac{1+t^{2}}{t^{2}+y_{n}^{2}}\,d\sigma_{n}(t)+\sigma_{n}([\mathbb{R}\cup\{\infty\}]\backslash[-T,T]).
\end{align*}
Since $y_{n}\to\infty$, the last integral above tends to zero as
$n\to\infty$, and thus
\[
\liminf_{n\to\infty}\sigma_{n}([\mathbb{R}\cup\{\infty\}]\backslash[-T,T])\ge1
\]
for every $T>0$. We conclude (using, for instance, the portmanteau
theorem) that $\alpha=\sigma(\{\infty\})\ge1$, thus reaching the
desired conclusion that $\lambda_{\psi}=\infty$.
\end{proof}

\section{\label{sec:applications} Applications}

All the results in this section arise from free probability considerations.
These connections are well-known and we refer to \cite{voic-fish1,voic-coalg,biane,bb-semigroups}
for their closer examination. The approach using Denjoy-Wolff points
was described earlier \cite{bb-new-approach}. The new features here
are continuity at the boundary and the two-variable results.
\begin{cor}
Suppose that $\varphi:\mathbb{D}\to\mathbb{D}$ is an analytic function.
There exists a unique continuous function $\omega:\overline{\mathbb{D}}\to\overline{\mathbb{D}}$
such that $\omega$ is analytic in $\mathbb{D}$, $\omega(0)=0$,
and
\[
\omega(z)=z\varphi(\omega(z)),\quad z\in\mathbb{D}.
\]
\end{cor}

\begin{proof}
For every $z\in\mathbb{\overline{D}}$ we define a map $\varphi_{z}:\mathbb{D}\to\mathbb{D}$
by
\[
\varphi_{z}(\lambda)=z\varphi(\lambda).
\]
Clearly, the map $z\mapsto\varphi_{z}(\lambda)$ is continuous for
fixed $\lambda$, and therefore Theorem \ref{thm:convergence theorem}
shows that the function $\omega(z)=\lambda_{\varphi_{z}}$ is continuous,
unless $\varphi_{z}$ is the identity function for some $z$. This
last situation can only occur when $z\in\mathbb{T}$ and $\varphi(\lambda)=\lambda/z$,
$z\in\mathbb{D}$. The corollary is proved by direct computation in
this special case. In all other cases, analyticity in $\mathbb{D}$
follows because $\omega(z)$ is the limit of the iterates $\varphi_{z}(\varphi_{z}(\cdots(\varphi_{z}(0))\cdots))$.
Uniqueness follows because $\varphi_{z}$ is not the identity map
for any $z\in\overline{\mathbb{D}}$.
\end{proof}
The preceding result, proved differently, is instrumental in the arguments
of \cite{bb-semigroups}, showing that the free (multiplicative) convolution
powers of a probability measure on the unit circle have certain regularity
properties. More precisely, such measures are absolutely continuous
relative to arclength measure, with the exception of a finite number
of atoms, and their densities are locally analytic where positive.
\begin{cor}
\label{cor:two variables-circle}Suppose that $\varphi_{1},\varphi_{2}:\mathbb{D}\to\mathbb{D}$
are analytic functions. Then there exist unique continuous functions
$\omega_{1},\omega_{2}:\overline{\mathbb{D}}\times\overline{\mathbb{D}}\to\overline{\mathbb{D}}$
that are analytic on $\mathbb{D}\times\mathbb{D}$, $\omega_{1}(0,0)=\omega_{2}(0,0)=0$,
and
\begin{align}
\omega_{1}(z_{1},z_{2}) & =z_{2}\varphi_{2}(\omega_{2}(z_{1},z_{2})),\label{eq:omega1}\\
\omega_{2}(z_{1},z_{2}) & =z_{1}\varphi_{1}(\omega_{1}(z_{1},z_{2})),\quad z_{1},z_{2}\in\mathbb{D}.\nonumber 
\end{align}
\end{cor}

\begin{proof}
Define $\omega_{2}(z_{1},z_{2})=\lambda_{\varphi_{z_{1},z_{2}}}$
and use (\ref{eq:omega1}) to define $\omega_{1}$, where $\varphi_{z_{1},z_{2}}(\lambda)=z_{1}\varphi_{1}(z_{2}\varphi_{2}(\lambda))$
for $\lambda\in\mathbb{D}$ and $z_{1},z_{2}\in\overline{\mathbb{D}}$.
One must worry again about the possibility that $\varphi_{z_{1},z_{2}}$
is the identity map for some $z_{1},z_{2}\in\overline{\mathbb{D}}.$
This can happen for only one pair $(z_{1},z_{2})\in\mathbb{T}^{2}$
and for functions $\varphi_{1},\varphi_{2}$ that are automorphisms
of $\mathbb{D}$. This case is again treated by direct computation;
see for instance \cite{serb-thesis}.
\end{proof}
For applications to free probability, a special case of the preceding
result is useful. For the proof, we just observe that the functions
$z\mapsto\omega_{j}(z,z)$ satisfy the requirements. Uniqueness of
these functions is an easy consequence of the uniqueness of Denjoy-Wolff
points.
\begin{cor}
Suppose that $\varphi_{1},\varphi_{2}:\mathbb{D}\to\mathbb{D}$ are
analytic functions. Then there exist unique continuous functions $\omega_{1},\omega_{2}:\overline{\mathbb{D}}\to\overline{\mathbb{D}}$
that are analytic on $\mathbb{D}$, $\omega_{1}(0)=\omega_{2}(0)=0$,
and
\[
\omega_{1}(z)\varphi_{1}(\omega_{1}(z))=\omega_{2}(z)\varphi_{2}(\omega_{2}(z))=\frac{\omega_{1}(z)\omega_{2}(z)}{z},\quad z\in\mathbb{D}\backslash\{0\}.
\]
\end{cor}

The existence of functions $\omega_{1},\omega_{2}$ defined in $\mathbb{D}$
and satisfying the conditions of the preceding corollary was first
proved in \cite{biane}. The new information here is that these functions
extend continuously to $\mathbb{T}$ (but see \cite{serb-thesis}
for the case in which the functions $\varphi_{j}$ continue analytically
through some arc and the continuations map that arc to $\mathbb{T}$.)
As before, this result can be used to study the regularity of free
multiplicative convolutions of Borel probability measures on $\mathbb{T}$;
see \cite{serb-thesis} for measures on $\mathbb{T}$ whose support
is not the entire circle.

For the counterparts of these results in the case of additive free
convolution, we recall the notation $\overline{\mathbb{H}}=\mathbb{H}\cup\mathbb{R}\cup\{\infty\}$
for the closure of the complex upper half-plane in the Riemann sphere.
\begin{cor}
Suppose that $\psi_{1},\psi_{2}:\mathbb{H}\to\mathbb{H}$ are two
analytic functions such that 
\[
\lim_{y\uparrow\infty}\frac{\psi_{j}(iy)}{iy}=0,\quad j=1,2.
\]
 Then there exist continuous functions 
\[
\omega_{1},\omega_{2}:(\mathbb{H}\cup\mathbb{R})\times(\mathbb{H}\cup\mathbb{R})\to\overline{\mathbb{H}}
\]
 that are finite and analytic on $\mathbb{H}\times\mathbb{H},$ and
\begin{align*}
\omega_{1}(z_{1},z_{2}) & =z_{2}+\psi_{2}(\omega_{2}(z_{1},z_{2})),\\
\omega_{2}(z_{1},z_{2}) & =z_{1}+\psi_{1}(\omega_{1}(z_{1},z_{2})),\quad z_{1},z_{2}\in\mathbb{H}.
\end{align*}
 
\end{cor}

The proof is almost identical with that of Corollary \ref{cor:two variables-circle},
using the family of maps 
\[
\varphi_{z_{1},z_{2}}(\lambda)=z_{1}+\psi_{1}(z_{2}+\psi_{2}(\lambda)),\quad\lambda\in\mathbb{H},z_{1},z_{2}\in\mathbb{H}\cup\mathbb{R}.
\]
If none of these maps is the identity on $\mathbb{H}$, Theorem \ref{thm:convergence theorem}
applies. If one of these maps is the identity, which can only occur
for one pair $(z_{1},z_{2})\in\mathbb{R}^{2}$ and only for fractional
linear maps $\psi_{j}$, the corollary is verified by a direct calculation
that can be found essentially in \cite{serb-thesis}. The asymptotic
condition on the functions $\psi_{j}$ ensures that $\omega_{j}$
is not equal to $\infty$ at some point in $\mathbb{H}\times\mathbb{H}$
(for instance $(i,i)$). The existence of the analytic functions $\omega_{j}(z,z)$,
$z\in\mathbb{H}$, was known earlier \cite{voic-fish1,biane,voic-coalg}. 

The maps $\omega_{j}$ can often be extended continuously to pairs
$(z_{1},z_{2})\in\overline{\mathbb{H}}\times\overline{\mathbb{H}}$
for which one or both coordinates are infinite. For instance, for
every $z_{2}\in\mathbb{\overline{H}}$ we have
\[
\lim_{z_{1}\to\infty}\varphi_{z_{1},z_{2}}(\lambda)=\infty,\quad\lambda\in\mathbb{H},
\]
and thus 
\[
\lim_{z_{1}\to\infty}\omega_{1}(z_{1},z_{2})=\lim_{z_{1}\to\infty}\lambda_{\varphi_{z_{1},z_{2}}}=\infty.
\]
If the limit
\[
\psi_{1}(\infty)=\lim_{\lambda\to\infty,\lambda\in\mathbb{H}}\psi_{1}(\lambda)
\]
 exists and is not zero, we have
\[
\lim_{z_{2}\to\infty}\varphi_{z_{1},z_{2}}(\lambda)=z_{1}+\psi_{1}(\infty),\quad\lambda\in\mathbb{H},
\]
so
\[
\lim_{z_{2}\to\infty}\omega_{1}(z_{1},z_{2})=z_{1}+\psi_{1}(\infty),\quad z_{1}\in\overline{\mathbb{H}}.
\]
Similar considerations apply to $\omega_{2}$. 

Setting $z_{1}=z_{2}$ in the above result yields the following result,
first proved in \cite{serb-leb,serb-boundedness}. 
\begin{cor}
\label{cor:halfplane one var}Suppose that $\psi_{1},\psi_{2}:\mathbb{H}\to\mathbb{H}$
are two analytic functions such that 
\[
\lim_{y\uparrow\infty}\frac{\psi_{j}(iy)}{iy}=0,\quad j=1,2.
\]
 Then there exist continuous functions $\omega_{1},\omega_{2}:\mathbb{H}\cup\mathbb{R}\to\mathbb{H}\cup\mathbb{R}$
that are finite and analytic on $\mathbb{H}$, and
\[
\omega_{1}(z)+\psi_{1}(\omega_{1}(z))=\omega_{2}(z)+\psi_{2}(\omega_{2}(z))=\omega_{1}(z)+\omega_{2}(z)-z,\quad z\in\mathbb{H}.
\]
\end{cor}

Of course, we also have
\[
\varangle\lim_{z\to\infty}\omega_{j}(z)=\infty,
\]
where $\varangle$ indicates a nontangential limit. This can be deduced
from the fact that
\[
\varangle\lim_{z\to\infty}\varphi_{z,z}(\lambda)=\infty,\lambda\in\mathbb{H},
\]
or, more simply, from the inequality $\Im\omega_{j}(z)\ge\Im z$.
Corollary \ref{cor:halfplane one var} is useful in the proof \cite{serb-leb,serb-boundedness}
that the free additive convolution $\mu_{1}\boxplus\mu_{2}$ of two
Borel probability measures on $\mathbb{R}$, both different from unit
point masses, is absolutely continuous relative to Lebesgue mesure,
except for finitely many atoms, and that the density is locally analytic
where positive. (The corresponding result for convolution powers is
in \cite{bb-semigroups}.) 

There is one more operation on measures to which our main result applies,
namely the free multiplicative convolution of Borel probability measures
on $[0,+\infty)$. The regularity of such free convolutions was examined
in \cite{bb-semigroups} and Theorem \ref{thm:convergence theorem}
provides an easier approach. The application of Theorem \ref{thm:convergence theorem}
involves the simply connected domain
\[
\Omega=\mathbb{C}\backslash[0,+\infty)
\]
 and only part of its prime end compactification. We recall that each
point $r\in(0,+\infty)$ corresponds to two distinct prime ends of
$\Omega$, denoted $r_{+}$ and $r_{-}$ which can be identified within
this compactification as
\[
r_{\pm}=\lim_{\theta\downarrow0}re^{\pm i\theta}.
\]
We consider the class $\mathcal{F}$ consisting of those analytic
functions $\eta:\Omega\to\mathbb{C}$ that satisfy the folllowing
conditions:
\begin{enumerate}
\item $\eta((-\infty,0))\subset(-\infty,0)$. In particular, $\overline{\eta(\overline{\lambda})}=\eta(\lambda)$
for every $\lambda\in\Omega$.
\item If $\lambda\in\mathbb{H}$, we have $\eta(\lambda)\in\mathbb{H}$
and $\arg(\eta(\lambda))\ge\arg\lambda,$ where $\arg\lambda$ represents
the principal value of the argument, that is, $\arg\lambda\in(0,\pi)$.
\item $\lim_{x\uparrow0}\eta(x)=0$.
\end{enumerate}
(The reader conversant with the notation of free probability may observe
that $\mathcal{F}$ consists of the eta-transforms of Borel probability
measures on $[0,+\infty)$. Point masses correspond to functions of
the form $\eta(\lambda)=\beta\lambda$ for some constant $\beta\ge0$.) 

It was shown in \cite{biane} (see also \cite{bb-new-approach}) that,
given two functions $\eta_{1},\eta_{2}\in\mathcal{F}$, not of the
form $\beta\lambda$, there exists a unique analytic functions $\omega\in\mathcal{F}$
such that, setting $f_{j}(\lambda)=\eta_{j}(\lambda)/\lambda$,
\begin{equation}
\omega(\lambda)=\lambda f_{1}(\lambda f_{2}(\omega(\lambda))),\quad\lambda\in\Omega.\label{eq:2}
\end{equation}
Corollary \ref{cor: positive line} provides a continuous extension
of $\omega$ to the closure of $\mathbb{H}$ in the Riemann sphere.
For the proof, we need a form of the Nevanlinna representation that
applies to the functions $f_{j}$ (see \cite[Lemma 2.2]{super convergence}):
there are finite, nonzero, positive Borel measures $\sigma_{1},\sigma_{2}$
on $(0,+\infty)$ and nonnegative constants $\beta_{1},\beta_{2}$
such that $\int_{(0,+\infty)}(1/t)\,d\sigma_{j}(t)<+\infty$, and
\begin{equation}
f_{j}(\lambda)=\frac{\eta_{j}(\lambda)}{\lambda}=\beta_{j}+\int_{(0,+\infty)}\frac{1+t^{2}}{t(t-\lambda)}\,d\sigma_{j}(t),\quad\lambda\in\Omega,\label{eq:Nevanlinna in Omega-1}
\end{equation}
for $j=1,2$. The continuous extension of the function $\omega$ to
$\mathbb{H}\cup[0,+\infty]$ is achieved by considering the family
of maps $(\varphi_{z})_{z\in\mathbb{H}\cup(0,+\infty)}$ defined by
\[
\varphi_{z}(\lambda)=zf_{1}(zf_{2}(\lambda)),\quad\lambda\in\mathbb{H}\cup(-\infty,0),z\in\mathbb{H}\cup(0,+\infty).
\]
 Clearly, the map $z\mapsto\varphi_{z}(\lambda)$ is continuous for
every $\lambda\in\mathbb{H}.$ We discuss the limit of these map as
$z\to0$ and $z\to\infty$ in order to also define maps $\varphi_{0}=0$
and $\varphi_{\infty}=\gamma/f_{2}$, where $\gamma=\lim_{x\downarrow-\infty}\eta_{2}(x)$.
When the limit $\gamma$ is infinite, we have $\varphi\equiv\infty$.
\begin{lem}
With the notation above, we have $\lim_{z\to0}\varphi_{z}(\lambda)=0$
for every $\lambda\in\mathbb{H}$. Moreover, there exists $\gamma\in\overline{\mathbb{H}}$
such that $\lim_{z\to\infty}\varphi_{z}(\lambda)=\gamma/f_{2}(\lambda)$
for every $\lambda\in\mathbb{H}$.
\end{lem}

\begin{proof}
We consider first the limit at $0$. Fix $\lambda\in\mathbb{H}$ and
suppose that $(z_{n})_{n\in\mathbb{N}}\subset\mathbb{H}\cup(0,+\infty)$
is a sequence converging to $0$. Suppose, in addition, that 
\[
\lim_{n\to\infty}\arg(z_{n})=\theta\in[0,\pi]
\]
exists. The points $w_{n}=z_{n}f_{2}(\lambda)$ tend to zero and their
arguments approximate $\theta+\arg f_{2}(\lambda)\in(0,2\pi)$. It
follows from (\ref{eq:Nevanlinna in Omega-1}) that $\lim_{n\to\infty}\eta_{1}(w_{n})=0$,
and this implies immediately that 
\[
\lim_{n\to\infty}\varphi_{z_{n}}(\lambda)=\lim_{n\to\infty}\frac{\eta_{1}(w_{n})}{f_{2}(\lambda)}=0.
\]
Compactness of the Riemann sphere then helps us to get rid of the
assumption on the arguments of $z_{n}$.

Consider next a sequence $(z_{n})_{n\in\mathbb{N}}\subset\mathbb{H}\cup(0,+\infty)$
that tends to infinity and $\lim_{n\to\infty}\arg(z_{n})=\theta\in[0,\pi]$
exists. Then the points $w_{n}=z_{n}f_{2}(\lambda)$ tend to infinity
and their arguments approximate $\theta+\arg f_{2}(\lambda)$. Now
(\ref{eq:Nevanlinna in Omega-1}) implies that $\lim_{n\to\infty}\eta_{1}(w_{n})=\lim_{x\downarrow-\infty}\eta_{1}(x)$.
If $\gamma=\lim_{x\downarrow-\infty}\eta_{1}(x)$ is infinite, we
conclude immediately that $\lim_{n\to\infty}\varphi_{z_{n}}(\lambda)=\infty$.
On the other hand, if $\gamma$ is finite, we have $\lim_{n\to\infty}\varphi_{z_{n}}(\lambda)=\gamma/f_{2}(\lambda)$.
\end{proof}
\begin{cor}
\label{cor: positive line}Suppose that $\eta_{1},\eta_{2}\in\mathcal{F}$
are as above, and let $\omega:\Omega\to\Omega$ be the analytic function
satisfying \emph{(\ref{eq:2})}. Then $\omega|\mathbb{H}$ has a continuous
extension to $\mathbb{H}\cup[0,+\infty]$. This extension satisfies
$\omega(0)=0$ and $\omega(\infty)=\infty$. 
\end{cor}

\begin{proof}
We use the family of maps 
\[
\varphi_{z}(\lambda)=zf_{1}(zf_{2}(\lambda)),\quad\lambda\in\mathbb{H},z\in\mathbb{H}\cup[0,+\infty].
\]
 considered in the preceding lemma. As in the preceding results, the
case in which one of these maps is the identity map\textemdash in
which case $z\in\mathbb{R}$\textemdash , is treated by direct calculation.
In general, a difficulty arises from the fact that $\varphi_{z}$
does not map $\mathbb{H}$ to itself. However, if we set 
\[
G_{z}=\{\lambda\in\mathbb{H}:\arg\lambda>\arg z\},
\]
we have $zf_{j}(\lambda)\in G_{z}$ for every $\lambda\in G_{z}$,
$j=1,2$, and therefore $\varphi_{z}(G_{z})\subset G_{z}$. Denote
by $u_{z}:\mathbb{H}\to G_{z}$ the conformal homeomorphism defined
by
\[
u_{z}(\lambda)=-(-\lambda)^{1-\frac{\theta}{\pi}},\lambda\in\mathbb{H},z=re^{i\theta}\in\mathbb{H}\cup(0,+\infty),
\]
where $\theta\in[0,\pi)$, $r>0,$ and the power is calculated using
the principal branch of the logarithm (that is, choosing $\arg(-\lambda)\in(-\pi,0)$
for $\lambda\in\mathbb{H})$. Then the map $\lambda\mapsto\psi_{z}(\lambda)=u_{z}(\varphi_{z}(u_{z}^{-1}(\lambda)))$,
sends $\mathbb{H}$ to $\overline{\mathbb{H}}$ and depends continuously
on $z$. Theorem \ref{thm:convergence theorem} shows that the map
$z\mapsto\lambda_{\psi_{z}}$ is continuous on $\mathbb{H}\cup[0,+\infty]$.
The corollary follows easily because $\omega(z)=u_{z}(\lambda_{\psi_{z}})$
for $z\in\mathbb{H}$, so $\omega(z)$ tends to $u_{r}(\lambda_{\psi_{z}})=\lambda_{\psi_{z}}$
as $\lambda\in\mathbb{H}$ tends to $r>0$.
\end{proof}

\end{document}